\documentclass[12pt, reqno]{amsart}

\author[P.~Leonetti]{Paolo Leonetti}
\address{
Department of Economics, 
Universit\`a degli Studi dell'Insubria, via Monte Generoso 71, Varese 21100, Italy}

\makeatletter
\@namedef{subjclassname@2020}{\textup{2020} Mathematics Subject Classification}
\makeatother


\title{On uniform summability}

\usepackage[T1]{fontenc}
\usepackage{amsmath}
\usepackage{amssymb}
\usepackage{amsthm}
\usepackage[left=2.5cm, right=2.5cm, top=3cm]{geometry}
\usepackage{hyperref}
\usepackage{fancyhdr}
\usepackage[inline]{enumitem}
\usepackage{comment}
\usepackage{nicefrac}
\usepackage{bm}
\usepackage{mathrsfs}
\usepackage{graphicx}
\usepackage[utf8]{inputenc}
\usepackage{cancel}
\usepackage{mathtools}
\usepackage{xcolor}

\newcommand{\vertiii}[1]{{\left\vert\kern-0.25ex\left\vert\kern-0.25ex\left\vert #1 
    \right\vert\kern-0.25ex\right\vert\kern-0.25ex\right\vert}}
		
\AtBeginDocument{%
   \def\MR#1{}
}

\newtheorem{thm}{Theorem}[section]
\newtheorem{cor}[thm]{Corollary}

\newtheorem{prop}[thm]{Proposition}

\theoremstyle{definition} 
\newtheorem{defi}[thm]{Definition}
\let\olddefi\defi
\renewcommand{\defi}{\olddefi\normalfont}
\newtheorem{example}[thm]{Example}
\let\oldexample\example
\renewcommand{\example}{\oldexample\normalfont}
\newtheorem{rmk}[thm]{Remark}
\let\oldrmk\rmk
\renewcommand{\rmk}{\oldrmk\normalfont}

\theoremstyle{remark}
\newtheorem{claim}{\textsc{Claim}}
\newtheorem*{claim*}{\textsc{Claim}}

\hypersetup{
    pdftitle={On uniform summability},
    pdfauthor={},
    pdfmenubar=false,
    pdffitwindow=true,
    pdfstartview=FitH,
    colorlinks=true,
    linkcolor=blue,
    citecolor=green
}

\providecommand{\MR}[1]{}

\providecommand{\MR}{\relax\ifhmode\unskip\space\fi MR }

\providecommand{\href}[2]{#2}

\begin{document}

\maketitle
\thispagestyle{empty}

\begin{abstract}
\noindent 
Let $\mathscr{A}$ be a nonempty set of infinite matrices of linear operators between two topological vector spaces. We show that a sequence is uniformly $\mathscr{A}$-summable if and only if it is $B$-summable for all matrices $B$ of linear operators such that the $n$th row of $B$ is the $n$th row for some $A \in \mathscr{A}$. This extends the main result of Bell in [Proc.~Amer.~Math.~Soc.~\textbf{38} (1973), 548--552]. 
We also provide several applications including uniform versions of Silverman--Toeplitz theorem, characterizations of almost regular matrices, uniform superior limits, and inclusion of ideal cores. Basically, our methods allow to translate ordinary results into their uniform versions, using directly the former ones. 
\end{abstract}

\section{Introduction and Main results}\label{sec:intro}

The aim of this work is to provide the connection between uniform summability and ordinary summability for infinite matrices of linear operators between topological vector spaces. Informally, this will allow to carry \textquotedblleft ordinary\textquotedblright\, statements into their \textquotedblleft uniform\textquotedblright\, versions, without repeating nor adapting the original proofs. 
We provide also in Section \ref{sec:applications} several applications which recover and extend many known results, including characterizations of almost regular matrices, uniform superior limits, and inclusions of cores for real sequences. 

To state more precisely our results, some notations are in order. 
Let $X,Y$ be topological vector spaces and denote by $\mathcal{L}(X,Y)$ the set of (not necessarily continuous) linear operators from $X$ into $Y$. 
Given an infinite matrix $A=(A_{n,k}: n,k \in \omega)$ of linear operators in $\mathcal{L}(X,Y)$, we write 
$$
\mathrm{dom}(A):=
\left\{\bm{x}=(x_0,x_1,\ldots) \in X^\omega: \sum\nolimits_k A_{n,k}x_k \text{ converges for all }n\in\omega\right\},
$$
where $\omega:=\{0,1,\ldots\}$. 
We also write $A\bm{x}:=(A_n\bm{x}: n \in \omega)$, where $A_n\bm{x}:=\sum\nolimits_k A_{n,k}x_k$ for each $n \in \omega$; hence, $A_n$ acts as a linear operator on $\mathrm{dom}(A)$ for each $n \in \omega$. A sequence $\bm{x}$ is said to be $A$\emph{-summable} if the transformed sequence $A\bm{x}$ (is well defined and) converges. We refer the reader to \cite{MR568707} for a textbook exposition on the theory of matrices of linear operators. 

Let $\mathscr{A}$ be a nonempty set of matrices of linear operators in $\mathcal{L}(X,Y)$, and define $\kappa:=|\mathscr{A}|$, so that we can index the matrices in $\mathscr{A}$ as $\{A^\nu: \nu<\kappa\}$. 
Let us also write $\mathrm{dom}(\mathscr{A}):=\bigcap_{\nu<\kappa}\mathrm{dom}(A^\nu)$,  
%
and
let $\mathscr{B}_\mathscr{A}$ be the set of all matrices $B=(B_{n,k})$ of linear operators in $\mathcal{L}(X,Y)$ such that the $n$th row of $B$ is the $n$th row of some $A \in \mathscr{A}$; equivalently, 
$B \in \mathscr{B}_\mathscr{A}$ 
if and only if 
for all $n \in \omega$ there exists $\nu<\kappa$ such that $B_{n,k}=A^\nu_{n,k}$ for all $k \in \omega$. 
It is clear from its definition that $\mathrm{dom}(\mathscr{A})=\bigcap \left\{\mathrm{dom}(B): B \in \mathscr{B}_{\mathscr{A}}\right\}$. 

Let also $\mathcal{I}$ be an ideal on $\omega$, that is, a family of subsets of $\omega$ which is stable under finite unions and subsets. Unless otherwise stated, we assume that $\omega\notin \mathcal{I}$ and that the family $\mathrm{Fin}$ of finite subsets of $\omega$ is contained in $\mathcal{I}$. We denote by $\mathcal{I}^\star:=\{\omega\setminus S: S \in \mathcal{I}\}$ its dual filter. 
A sequence $\bm{x} \in X^\omega$ is $\mathcal{I}$\emph{-convergent to} $\eta\in X$, shortened as $\mathcal{I}\text{-}\lim \bm{x}=\eta$, if $\{n \in \omega: x_n \notin U\} \in \mathcal{I}$ for every neighborhood $U$ of $\eta$ (of course, if $\mathcal{I}=\mathrm{Fin}$, this coincides with ordinary convergence). 
An ideal $\mathcal{I}$ is \emph{countably generated} if there exists a sequence $(S_n)$ in $\mathcal{I}$ such that, for all $S\subseteq \omega$, 
$S \in \mathcal{I}$ if and only if $S\subseteq \bigcup_{n \in F}S_n$ for some $F \in \mathrm{Fin}$. We refer to \cite{MR1711328} for an excellent monograph on the theory of ideals on $\omega$ (the reader who is not interested in results involving ideals can assume hereafter that the latter ones are always $\mathrm{Fin}$). 

\begin{defi}\label{defi:uniform}
Let $\mathscr{A}$ be a nonempty set of matrices of linear operators in $\mathcal{L}(X,Y)$, where $X,Y$ are topological vector spaces. Let also $\mathcal{I}$ be an ideal on $\omega$. A sequence $\bm{x}\in X^\omega$ is said to be \emph{uniformly }$\mathscr{A}$\emph{-summable} to a vector $\eta \in Y$, shortened as
$$
\mathcal{I}\text{-}\lim \mathscr{A}\bm{x}=\eta,
$$
if $\bm{x} \in \mathrm{dom}(\mathscr{A})$ and 
$\mathcal{I}\text{-}\lim A\bm{x}=\eta$ uniformly on $A \in \mathscr{A}$. 
\end{defi}
Equivalently, 
this 
means that, for each neighborhood $U$ of $\eta$, there exists $S \in \mathcal{I}^\star$ such that $A_n\bm{x} \in U$ for all $n \in S$ and all $A \in \mathscr{A}$. 
Of course, if $\mathscr{A}=\{A\}$ then $\mathcal{I}\text{-}\lim \mathscr{A}\bm{x}=\eta$ is equivalent to the classical $\mathcal{I}\text{-}\lim A\bm{x}=\eta$. 
As remarked by Bell in \cite{MR310489}, the above definition in the case $X=Y=\mathbf{R}$ and $\mathcal{I}=\mathrm{Fin}$ corresponds to notions already studied in the literature, e.g., $F_A$-summability by Lorentz \cite{MR27868}, $A_B$ summability by Mazhar and Siddiqi \cite{MR259482}, or order summability $[g]$ by Jurkat and Peyerimhoff \cite{MR282142}; very recently, it appeared also as [weakly] almost convergence in normed spaces by Karaku{\c{s}} and Ba{\c{s}}ar, see \cite[Definition 1.1]{MR4042592}.

With the above premises, we are going to show that a sequence is uniformly $\mathscr{A}$-summable if and only if it is $B$-summable for all matrices $B \in \mathscr{B}_{\mathscr{A}}$.  
%
\begin{thm}\label{thm:characterizationuniformconvergence}
Let $X,Y$ be topological vector spaces such that $Y$ is separable, and fix nonempty set $\mathscr{A}$ of matrices $A=(A_{n,k})$ of linear operators in $\mathcal{L}(X,Y)$. 
Let also $\mathcal{I}$ be a countably generated ideal on $\omega$ and pick a sequence $\bm{x} \in \mathrm{dom}(\mathscr{A})$. 
Then the following are equivalent\textup{:}
\begin{enumerate}[label={\rm (\roman*)}]
\item \label{item:1uniformcondition} There exists $\eta_1 \in Y$ such that $\mathcal{I}\text{-}\lim \mathscr{A}\bm{x}=\eta_1$\textup{;}
\item \label{item:2uniformcondition} For each $B \in \mathscr{B}_\mathscr{A}$ there exists $\eta_B\in Y$ such that $\mathcal{I}\text{-}\lim B\bm{x}=\eta_B$\textup{;}
\item \label{item:3uniformcondition} There exists $\eta_2 \in Y$ such that $\mathcal{I}\text{-}\lim B\bm{x}=\eta_2$ for all $B \in 
\mathscr{B}_{\mathscr{A}}$\textup{.}
\end{enumerate}
In addition, in such case, $\eta_1=\eta_2$. 
\end{thm}

We remark that the case where $X=Y=\mathbf{R}$, $\mathcal{I}=\mathrm{Fin}$, and $\kappa=\omega$ has been proved by Bell in \cite[Theorem 1]{MR310489}.
In the same article, Bell proves its usefulness in the context of Fourier effectiveness and strongly regular matrices, cf. also \cite{MR427890, MR358128}. We remark also that the class of countably generated ideals on $\omega$ includes, up to isomorphism, the ideal $\mathrm{Fin}$, the Fubini product $\emptyset \times \mathrm{Fin}$, and the Fubini sum $\mathrm{Fin}\oplus \mathcal{P}(\omega)$; see \cite[Proposition 1.2.8]{MR1711328} for details and cf. also \cite[Section 2]{MR3543775} and \cite[Remark 2.16]{Leo22}.

We discuss in Section \ref{sec:applications} several applications of Theorem \ref{thm:characterizationuniformconvergence}, mainly in operator versions. This will extend and provide a unified framework for many known characterizations. 


The proofs of Theorem \ref{thm:characterizationuniformconvergence} and  the next results are given in Section \ref{sec:proofs}. 

\section{Applications and Related Results}\label{sec:applications}

\subsection{A uniform Silverman--Toeplitz theorem} 
Let us suppose hereafter that $X$ and $Y$ are Banach spaces and endow $\mathcal{L}(X,Y)$ (and all its subspaces) with the strong operator topology, so that $\lim_n T_n=T$ means that $(T_nx: n \in \omega)$ is convergent in the norm of $Y$ to $Tx$ for all $x \in X$. The subspace of bounded linear operators of $\mathcal{L}(X,Y)$ is denoted by $\mathcal{B}(X,Y)$. 

Given an infinite matrix $A=(A_{n,k}: n,k \in \omega)$ of linear operators  $A_{n,k} \in \mathcal{L}(X,Y)$, for each $n \in \omega$ and $E\subseteq \omega$ we write $A_{n,E}:=(A_{n,k}: k \in E)$, and define the group norm 
$$
\|A_{n,E}\|:=\sup\left\{\left\|\sum\nolimits_{k \in F}A_{n,k}x_k\right\|: F\subseteq E \text{ is finite and each }x_k \in B_X\right\}, 
$$
where $B_X$ stands for the closed unit ball of $X$, cf. e.g.  \cite{MR0447877, MR568707}. If $E=\omega\setminus k$, we shorten $A_{n,E}$ with $A_{n,\ge k}$ (as usual $k \in \omega$ is identified with the set $\{0,1,\ldots,k-1\}$). 
Note that, if $X=Y=\mathbf{R}$ and each $A_{n,k}$ is represented by  the real number $a_{n,k}$, then $\|A_{n,E}\|=\sum_{k \in E}|a_{n,k}|$.  

For each sequence subspace $\mathcal{A}\subseteq X^\omega$ and $\mathcal{B}\subseteq Y^\omega$, let $(\mathcal{A}, \mathcal{B})$ be the set of matrices $A=(A_{n,k}: n,k \in \omega)$ of 
linear operators in $\mathcal{L}(X,Y)$ such that 
$$
\mathcal{A} \subseteq \mathrm{dom}(A)
\,\,\,\text{ and }\,\,\,
A\bm{x} \in \mathcal{B} \,\text{ for all }\bm{x} \in \mathcal{A}.
$$ 

Moreover, given an ideal $\mathcal{I}$ on $\omega$, a sequence $\bm{x}\in X^\omega$ is said to be $\mathcal{I}$-bounded if $\{n \in \omega: \|x_n\|\ge \kappa\}\in \mathcal{I}$ for some $\kappa \in \mathbf{R}$. The vector space of $\mathcal{I}$-bounded sequences is denoted by $\ell_\infty(X, \mathcal{I})$, we write $c(X,\mathcal{I})$, $c_0(X,\mathcal{I})$, and $c_{00}(X,\mathcal{I})$ for the subspaces of $\mathcal{I}$-convergent sequences, $\mathcal{I}$-convergent to $0$ sequences, and sequences which are supported on some elements of $\mathcal{I}$, respectively (of course, if $\mathcal{I}=\mathrm{Fin}$ and $X=\mathbf{R}$, they correspond to the classical sequence spaces $c$, $c_0$, and $c_{00}$, resp.). In each of these subspaces, a superscript $b$ denotes its intersection with $\ell_\infty(X)$, e.g. $c_{00}^b(X,\mathcal{I}):=c_{00}(X,\mathcal{I})\cap \ell_\infty(X)$; cf. \cite[Section 2]{Leo22} for further details.

The next definition corresponds to the uniform version of \cite[Definition 2.1]{Leo22}:
\begin{defi}\label{defi:IJregTfamily}
Let $X,Y$ be topological vector spaces and fix $T \in \mathcal{L}(X,Y)$. Let also $\mathcal{I}, \mathcal{J}$ be ideals on $\omega$. A nonempty set of matrices $\mathscr{A}$
is said to be $(\mathcal{I}, \mathcal{J})$\emph{-regular with respect to} $T$ if $A \in (c^b(X,\mathcal{I}), c^b(Y,\mathcal{J}))$ for all $A \in \mathscr{A}$ and 
\begin{equation}\label{eq:uniformAIJregular}
\mathcal{J}\text{-}\lim A\bm{x}=T(\mathcal{I}\text{-}\lim \bm{x}) \,\,
\text{ uniformly on }A \in \mathscr{A}
\end{equation}
for all bounded $\mathcal{I}$-convergent sequences $\bm{x} \in c^b(X,\mathcal{I})$. 
\end{defi}
Of course, if $X=Y=\mathbf{R}$, $\mathcal{I}=\mathcal{J}=\mathrm{Fin}$, $T$ is the identity operator, and $\mathscr{A}=\{A\}$, then the above notion coincides with the classical regularity of $A$, see e.g. \cite[Chapter 7]{MR0390692}. 

In the case where $\mathscr{A}$ is a singleton, we recall the following characterization: 
\begin{thm}\label{thm:IJREGULARold}
Fix a linear operator $T \in \mathcal{B}(X,Y)$, where $X,Y$ are Banach spaces. 
Let also $\mathcal{I}, \mathcal{J}$ be ideals on $\omega$ and let $A=(A_{n,k})$ be a matrix of bounded linear operators in $\mathcal{B}(X,Y)$. 

Then $\{A\}$ is $(\mathcal{I}, \mathcal{J})$-regular with respect to $T$ if and only if\textup{:} 
\begin{enumerate}[label={\rm (\textsc{T}\arabic{*})}]
\item \label{item:T1} $\sup_n\|A_{n,\omega}\|<\infty$\textup{;}
\end{enumerate}
\begin{enumerate}[label={\rm (\textsc{T}\arabic{*})}] 
\setcounter{enumi}{2}
\item \label{item:T3} $c^b(X,\mathcal{I})\subseteq \mathrm{dom}(A)$\textup{;}
\end{enumerate}
\begin{enumerate}[label={\rm (\textsc{T}\arabic{*})}] 
\setcounter{enumi}{3}
\item \label{item:T4} $\mathcal{J}\text{-}\lim_n \sum_{k}A_{n,k}=T$\textup{;}
\end{enumerate}
\begin{enumerate}[label={\rm (\textsc{T}\arabic{*})}] 
\setcounter{enumi}{4}
\item \label{item:T5} $\mathcal{J}\text{-}\lim A\bm{x}=0$ for all $\bm{x} \in c_{00}^b(X,\mathcal{I})$
\textup{.} 
\end{enumerate}
\end{thm}
\begin{proof}
It follows by \cite[Theorem 2.5]{Leo22} in the case where each $A_{n,k}$ is a bounded linear operator (hence, item (T2) therein would be void). 
\end{proof}

We remark that equivalent versions of items \ref{item:T3} and \ref{item:T5} which depend only on the entries $A_{n,k}$ can be found in \cite{Leo25}. 

Using both Theorem \ref{thm:characterizationuniformconvergence} and Theorem \ref{thm:IJREGULARold}, we can characterize the $(\mathcal{I}, \mathcal{J})$-regularity of a family $\mathscr{A}$, provided that $\mathcal{J}$ is countably generated: 
\begin{thm}\label{thm:uniformtoeplitz}
Let $X,Y$ be Banach spaces, let $\mathscr{A}$ be a nonempty set of matrices of bounded linear operators in $\mathcal{B}(X,Y)$, and fix $T \in \mathcal{B}(X,Y)$. Let $\mathcal{I}, \mathcal{J}$ be ideals on $\omega$ such that $\mathcal{J}$ is countably generated. 

Then $\mathscr{A}$ is $(\mathcal{I}, \mathcal{J})$-regular with respect to $T$ if and only if\textup{:} 
\begin{enumerate}[label={\rm (\textsc{U}\arabic{*})}]
\item \label{item:U1} $\sup_n\|A_{n,\omega}\|<\infty$ for all $A \in \mathscr{A}$\textup{;}
\end{enumerate}
\begin{enumerate}[label={\rm (\textsc{U}\arabic{*})}] 
\setcounter{enumi}{1}
\item \label{item:U2} $\sup_{A \in \mathscr{A}} \sup_{n \in J}\|A_{n,\omega}\|<\infty$ for some $J \in \mathcal{J}^\star$\textup{;}
\end{enumerate}
\begin{enumerate}[label={\rm (\textsc{U}\arabic{*})}] 
\setcounter{enumi}{2}
\item \label{item:U3} $c^b(X,\mathcal{I})\subseteq \mathrm{dom}(A)$ for all $A \in \mathscr{A}$\textup{;}
\end{enumerate}
\begin{enumerate}[label={\rm (\textsc{U}\arabic{*})}] 
\setcounter{enumi}{3}
\item \label{item:U4} $\mathcal{J}\text{-}\lim_n \sum_{k}A_{n,k}=T$ uniformly on $A \in \mathscr{A}$\textup{;} 
\end{enumerate}
\begin{enumerate}[label={\rm (\textsc{U}\arabic{*})}] 
\setcounter{enumi}{4}
\item \label{item:U5} $\mathcal{J}\text{-}\lim A\bm{x}=0$ uniformly on $A \in \mathscr{A}$ for all $\bm{x} \in c_{00}^b(X,\mathcal{I})$\textup{.} 
\end{enumerate}
\end{thm}

In the case where both $X$ and $Y$ are finite dimensional, we obtain a more explicit characterization. 
 For, suppose that $X=\mathbf{R}^d$ and $Y=\mathbf{R}^m$ for some integers $d,m \ge 1$. 
Then, if $\mathscr{A}=\{A^\nu: \nu<\kappa\}$ is an arbitrary family, each linear operator $A^\nu_{n,k}$ is represented by the real matrix $[a^\nu_{n,k}(i,j): 1\le i\le m, 1\le j\le d\,]$ and $T$ is represented by the real matrix $[t(i,j): 1\le i\le m, 1\le j\le d\,]$; of course, $T$ and all $A_{n,k}^\nu$ are continuous. 

\begin{thm}\label{cor:maincharacterizationIJregular}
Suppose that $X=\mathbf{R}^d$ and $Y=\mathbf{R}^m$. Let $\mathscr{A}=\{A^\nu: \nu<\kappa\}$ be a nonempty set of matrices of linear operators in $\mathcal{B}(X,Y)$ and fix a linear operator $T \in \mathcal{B}(X,Y)$. Let also $\mathcal{I}, \mathcal{J}$ be ideal on $\omega$ such that $\mathcal{J}$ is countably generated. 

Then $\mathscr{A}$ is $(\mathcal{I}, \mathcal{J})$-regular with respect to $T$ if and only if\textup{:}
\begin{enumerate}[label={\rm (\textsc{D}\arabic{*})}]
\item \label{item:D1} $\sup_n\sum_k \sum_{i,j} |a_{n,k}^\nu (i,j)|<\infty$ for each $\nu$\textup{;}
\item \label{item:D2} $\sup_\nu \sup_{n \in J} \sum_k \sum_{i,j} |a^\nu_{n,k} (i,j)|<\infty$ for some $J \in \mathcal{J}^\star$\textup{;}
\item \label{item:D3} 
$\mathcal{J}\text{-}\lim_n \sum_k a^\nu_{n,k}(i,j)=t(i,j)$ uniformly on $\nu$, for all $1\le i\le m$ and $1\le j\le d$\textup{;}
\item \label{item:D4} $\mathcal{J}\text{-}\lim_n \sum_{k \in E}a^\nu_{n,k}(i,j)=0$ uniformly on $\nu$,  for all $1\le i\le m$, $1\le j\le d$, and 
$E \in \mathcal{I}$\textup{.} 
\end{enumerate}
\end{thm}

We remark that the proof of the above corollary in the case where $\mathscr{A}$ is a singleton can be deduced also from \cite[Corollary 2.11]{Leo22}. 

In particular, in the case $d=m=1$ and $T$ is the identity operator $I$, we obtain the ideal version of the uniform Silverman--Toeplitz theorem given by Bell in \cite[Theorem 2]{MR310489}. 
We provide below this special case for future references (here, we write $a^\nu_{n,k}:=a^\nu_{n,k}(1,1)$ for all $n,k \in \omega$ and $\nu<\kappa$): 
\begin{cor}\label{cor:maincorollaryuniform}
Suppose that $X=Y=\mathbf{R}$ and let $\mathcal{I}$, $\mathcal{J}$ be ideals on $\omega$ such that $\mathcal{J}$ is countably generated. 
Let also $\mathscr{A}=\{A^\nu: \nu<\kappa\}$ be a 
family of infinite real matrices. 

Then $\mathscr{A}$ is $(\mathcal{I}, \mathcal{J})$-regular with respect to $I$ if and only if\textup{:}
\begin{enumerate}[label={\rm (\textsc{M}\arabic{*})}]
\item \label{item:M1new} $\sup_n\sum_k |a^\nu_{n,k}|<\infty$ for each $\nu$\textup{;}
\item \label{item:M1new} $\sup_\nu \sup_{n \in J}\sum_k |a^\nu_{n,k}|<\infty$ for some $J \in \mathcal{J}^\star$\textup{;}
\item \label{item:M3new} 
$\mathcal{J}\text{-}\lim_n \sum_k a^\nu_{n,k}=1$ uniformly on $\nu$\textup{;}
\item \label{item:M4new} $\mathcal{J}\text{-}\lim_n \sum_{k \in E}a^\nu_{n,k}=0$ uniformly on $\nu$, for all $E \in \mathcal{I}$\textup{.} 
\end{enumerate}
\end{cor}
Note that the case where $\mathcal{I}=\mathcal{J}=\mathrm{Fin}$ and $\mathscr{A}$ is a singleton corresponds to the classical Silverman--Toeplitz theorem, see e.g. \cite[Chapter 7, Theorem 3]{MR0390692}. 

As another application, we characterize the class of matrices which maps every bounded sequences uniformly to zero. 
\begin{cor}\label{cor:maptozero}
Suppose that $X=\mathbf{R}^d$ and $Y=\mathbf{R}^m$. Let $\mathscr{A}=\{A^\nu: \nu<\kappa\}$ be a nonempty family of matrices of linear operators in $\mathcal{B}(X,Y)$. Let also $\mathcal{J}$ be a countably generated ideal on $\omega$. 

Then $\mathscr{A}\subseteq (\ell_\infty(X), c_0^b(Y))$ and satisfies
\begin{equation}\label{eq:uniformconditionwrgwg}
\forall \bm{x}\in \ell_\infty(X), \quad 
\mathcal{J}\text{-}\lim A\bm{x}=0 \quad \text{ uniformly on }\nu
\end{equation}
if and only if items \ref{item:D1} and \ref{item:D2} hold, together with\textup{:}
\begin{enumerate}[label={\rm (\textsc{D}\arabic{*}$^\sharp$)}] 
\setcounter{enumi}{2}
\item 
\label{item:D3sharp}
$\mathcal{J}\text{-}\lim_n \sum_k |a^\nu_{n,k}(i,j)|=0$ uniformly on $\nu$, for all $1\le i\le m$ and $1\le j\le d$\textup{.}
\end{enumerate}
\end{cor}
The special case where $d=m=1$, $\mathcal{J}=\mathrm{Fin}$, and $\mathscr{A}$ countably infinite can be found in \cite[Theorem 4]{MR310489}. A characterization of the matrix class $(\ell_\infty(X), c_0^b(Y))$ (hence, the case where $\mathscr{A}$ is a singleton) can be found in \cite[Corollary 2.12]{Leo22}, which in turn extends \cite[Lemma 3.2]{MR2209588} in the real case for a particular choice of $\mathcal{J}$.

\subsection{$\sigma$-means and almost convergence} 
Let $X$ be a Banach space and fix an injective map $\sigma: \omega\to \omega$ without periodic points (i.e., it satisfies $\sigma^{p+1}(n) \neq n$ for all $n,p \in \omega$). Let $\mathscr{F}^\sigma=\{F^{\sigma,\nu}: \nu\in \omega\}$ be the family of matrices of bounded linear operators in $\mathcal{B}(X,X)$ defined as follows: for all $n,k,\nu \in \omega$, set 
$$
F^{\sigma,\nu}_{n,k}:=\frac{I}{n+1} \quad \text{ if }k \in \{\sigma(\nu), \sigma(\nu+1), \ldots, \sigma(\nu+n)\}
$$
and $F^{\sigma,\nu}_{n,k}:=0$ otherwise (recall that $I$ stands for the identity operator on $X$, so that $F^{\sigma,\nu}_n\bm{x}=\sum_{i=0}^n x_{\sigma(\nu+i)}/(n+1)$ for all $n,\nu \in \omega$ and all sequences $\bm{x} \in X^\omega$). 

\begin{defi}\label{defi:ac}
Let $X$ and $\sigma$ be as above. Let also $\mathcal{I}$ be an ideal on $\omega$. A sequence $\bm{x} \in X^\omega$ is said to be $(\mathcal{I},\sigma)$\emph{-convergent} to a vector $\eta \in X$, shortened as 
$$
(\mathcal{I}, \sigma)\text{-}\lim \bm{x}=\eta,
$$ 
if $\mathcal{I}\text{-}\lim \mathscr{F}\bm{x}=\eta$. If $\mathcal{I}=\mathrm{Fin}$, we simply write $\sigma\text{-}\lim \bm{x}=\eta$. 

The space of bounded $(\mathcal{I},\sigma)$-convergent sequences is denoted by  $V_{\mathcal{I},\sigma}(X)$. 
\end{defi}
If $X=\mathbf{R}$ and $\mathcal{I}=\mathrm{Fin}$, this notion coincides with $\sigma$-convergence studied e.g. by Raimi \cite{MR154005} and Schaefer \cite{MR306763}. If, in addition, one chooses 
\begin{equation}\label{eq:sigma0}
\forall n \in \omega, \quad \sigma_0(n):=n+1,
\end{equation}
then $V_{\mathrm{Fin},\sigma_0}(\mathbf{R})$ coincides with the space of almost convergent sequences (it is usually denoted by $\mathrm{ac}$); see also \cite{MR235340} for almost convergence of the transformed sequences $A\bm{x}$. In a Banach space setting, the space $V_{\mathrm{Fin},\sigma_0}(X)$ appears also in \cite[Definition 1.1]{MR4042592}. 

In addition, if $\mathcal{I}$ is a suitable copy of the Fubini product $\emptyset \times \mathrm{Fin}$ (which is a countably generated ideal), then $\mathcal{I}$-convergence can be seen as the Pringsheim convergence of double sequences $(x_{j,k}: j,k \in \omega)$, see \cite[Section 4.2]{MR3955010}. Accordingly, adapting the definition of $\mathscr{F}$ above, the notion of $(\mathcal{I},\sigma)$-convergence coincides with almost convergence of double sequences introduced by M\'{o}ricz and Rhoades in \cite{MR948914}.

It is worth noting that $\sigma$-convergence cannot be rewritten as $\mathcal{J}$-convergence, for any ideal $\mathcal{J}$ on $\omega$. Indeed, set $x:=(1,0,1,0,\ldots)$. Then $\sigma\text{-}\lim x=\nicefrac{1}{2}$ while $\mathcal{J}\text{-}\lim x$ is either $0$ or $1$ or it does not exist, for every ideal $\mathcal{J}$.

\begin{defi}\label{defi:almostregular}
Let $X,Y$ be Banach spaces. Let also $\mathcal{I}, \mathcal{J}$ be ideals on $\omega$ and fix an injection $\sigma: \omega \to \omega$ without periodic points. A matrix $A=(A_{n,k})$ of linear operators in $\mathcal{L}(X,Y)$ is said to be $(\mathcal{I}, \mathcal{J}, \sigma)$\emph{-almost regular with respect to $T \in \mathcal{L}(X,Y)$} if 
$$
A \in (c^b(X,\mathcal{I}), V_{\mathcal{J},\sigma}(Y))\quad \text{ and }\quad 
(\mathcal{J},\sigma)\text{-}\lim A\bm{x}=T(\mathcal{I}\text{-}\lim \bm{x})
\,\text{ for all }\bm{x} \in c^b(X,\mathcal{I}).
$$
\end{defi}
In the case where $X=Y=\mathbf{R}$, $\mathcal{I}=\mathcal{J}=\mathrm{Fin}$, and $T=I$, real matrices $A$ which satisfy the above notion have been called \textquotedblleft$\sigma$-regular\textquotedblright\,by Schaefer in \cite[Definition 2]{MR306763}. If, in addition, $\sigma=\sigma_0$ as defined in \eqref{eq:sigma0} then they are commonly known as \textquotedblleft almost regular,\textquotedblright\, see e.g. \cite{MR310489} or \cite[Definition 2.3]{MR201872}. 

A closer look at Definition \ref{defi:almostregular} reveals that a matrix $A$ is $(\mathcal{I}, \mathcal{J}, \sigma)$-almost regular with respect to $T$ if and only if the family $\mathscr{A}:=\{F^{\sigma,\nu}A: \nu\in \omega\}$ is $(\mathcal{I}, \mathcal{J})$-regular with respect to $T$, according to Definition \ref{defi:IJregTfamily}. 

Hence, using Theorem \ref{thm:uniformtoeplitz}, we obtain immediately a characterization of $(\mathcal{I}, \mathcal{J}, \sigma)$-almost regularity in a Banach space setting (we leave details to the interested reader). 
However, in the finite dimension, the latter characterization simplifies: 
\begin{thm}\label{thm:kingmultidimensional}
Suppose that $X=\mathbf{R}^d$ and $Y=\mathbf{R}^m$, and let $A=(A_{n,k})$ be a matrix of bounded linear operators in $\mathcal{B}(X,Y)$. Fix also $T \in \mathcal{B}(X,Y)$, an injection $\sigma: \omega\to \omega$ 
without periodic points, 
and let $\mathcal{I}, \mathcal{J}$ be ideals on $\omega$ with $\mathcal{J}$ countably generated. 


Then $A$ is $(\mathcal{I}, \mathcal{J}, \sigma)$-almost regular with respect to $T$ if and only if\textup{:}
\begin{enumerate}[label={\rm (\textsc{K}\arabic{*})}]
\item \label{item:K1} $\sup_n\sum_k \sum_{i,j} |a_{n,k}(i,j)|<\infty$\textup{;}
\end{enumerate}
\begin{enumerate}[label={\rm (\textsc{K}\arabic{*})}]
\setcounter{enumi}{1}
\item \label{item:K2} 
$(\mathcal{J},\sigma)\text{-}\lim_n \sum_k a_{n,k}(i,j)=t(i,j)$, for all $1\le i\le m$ and $1\le j\le d$\textup{;}
\end{enumerate}
\begin{enumerate}[label={\rm (\textsc{K}\arabic{*})}]
\setcounter{enumi}{2}
\item \label{item:K3} $(\mathcal{J},\sigma)\text{-}\lim_n \sum_{k \in E}a_{n,k}(i,j)=0$, for all $1\le i\le m$, $1\le j\le d$, and 
$E \in \mathcal{I}$\textup{.} 
\end{enumerate}
\end{thm}
Rather special instances of Theorem \ref{thm:kingmultidimensional} already appeared in the literature. For example, we deduce Schaefer's characterization  \cite[Theorem 2]{MR306763} in the case $d=m=1$ and $\mathcal{I}=\mathcal{J}=\mathrm{Fin}$. If, in addition, $\sigma=\sigma_0$ we obtain King's characterization of almost regular matrices \cite[Theorem 3.2]{MR201872}.

\subsection{Uniform superior limits and core inclusions} 
In this last section, we specialize to the real case (hence $X=Y=\mathbf{R}$).  Given a bounded real sequence $\bm{x}=(x_n: n \in \omega) \in \ell_\infty$ and an ideal $\mathcal{I}$ on $\omega$, we say that $\eta \in \mathbf{R}$ is a $\mathcal{I}$\emph{-cluster point of} $\bm{x}$ if 
$$
\forall \varepsilon>0, \quad 
\{n \in \omega: |x_n-\eta|\le \varepsilon\}\notin \mathcal{I}. 
$$
The set of $\mathcal{I}$-cluster points of $\bm{x}$ is denoted by $\Gamma_{\bm{x}}(\mathcal{I})$. It is known that, for each $\bm{x} \in \ell_\infty$, the set $\Gamma_{\bm{x}}(\mathcal{I})$ is nonempty and compact; see e.g. \cite[Lemma 3.1]{MR3920799}. Hence, we can define 
$$
\mathcal{I}\text{-}\liminf \bm{x}:=\min \Gamma_{\bm{x}}(\mathcal{I})
\quad \text{ and }\quad 
\mathcal{I}\text{-}\limsup \bm{x}:=\max \Gamma_{\bm{x}}(\mathcal{I})
$$
for each $\bm{x} \in \ell_\infty$. As one may expect, we have that the above values coincide if and only if $\bm{x}$ is $\mathcal{I}$-convergent, see e.g. \cite[Corollary 3.4]{MR3920799}. 

At this point, let us recall that an infinite real matrix $A=(a_{n,k}: n,k \in \omega)$ maps bounded sequences into bounded sequences (i.e., $A \in (\ell_\infty, \ell_\infty)$ if and only if $\|A\|:=\sup_n \sum_k |a_{n,k}|<\infty$, see e.g. \cite[Theorem 2.3.5]{MR1817226}. The following result, in the same spirit as Theorem \ref{thm:characterizationuniformconvergence}, provides a characterization of uniform superior limits.  
\begin{prop}\label{prop:uniflimsup}
Let $\mathscr{A}$ be a nonempty family of real matrices $A=(A_{n,k}: n,k \in \omega)$ such that $\sup\{\|A\|: A \in \mathscr{A}\}<\infty$. Let also $\mathcal{I}$ be an ideal on $\omega$. Then
\begin{equation}\label{eq:claimuniflimsup}
\forall \bm{x} \in \ell_\infty, \quad 
\mathcal{I}\text{-}\limsup_{n\to \infty}\,  \sup_{A \in \mathscr{A}} A_n\bm{x}
=\sup_{B \in \mathscr{B}_{\mathscr{A}}}\, \mathcal{I}\text{-}\limsup_{n\to \infty} B_n\bm{x}.
\end{equation}
\end{prop}

Lastly, we are going to show that Proposition \ref{prop:uniflimsup} is useful to recover \textquotedblleft uniform\textquotedblright\,results from ordinary ones (in particular, without repeating the techniques of the latter proofs). For instance, recall that $\mathcal{I}$\emph{-core} of a bounded real sequence $\bm{x}\in \ell_\infty$ is the interval
$$
\mathrm{core}_{\bm{x}}(\mathcal{I}):=[\mathcal{I}\text{-}\liminf \bm{x}, \mathcal{I}\text{-}\limsup \bm{x}].
$$
In the case where $\mathcal{I}=\mathrm{Fin}$, $\mathrm{core}_{\bm{x}}(\mathcal{I})$ is usually called \emph{Knopp core}. 
We refer the reader to \cite{MR4126774, MR3955010} for several properties and characterizations of $\mathcal{I}$-cores for sequences taking values in topological vector spaces. 

At this point, let us recall the following characterization about inclusions of ideal cores, cf. also \cite{MR2241135, MR1416085} and references therein for related results. 
\begin{thm}\label{thm:connor}
Let $\mathcal{I}$ be an ideal on $\omega$ and pick a matrix $A \in (\ell_\infty, \ell_\infty)$. Then 
\begin{equation}\label{eq:inclusioncores}
\forall \bm{x} \in \ell_\infty, \quad 
\mathrm{core}_{A\bm{x}}(\mathrm{Fin})\subseteq \mathrm{core}_{\bm{x}}(\mathcal{I})
\end{equation}
if and only if the following conditions hold\textup{:}
\begin{enumerate}[label={\rm (\textsc{C}\arabic{*})}]
\item \label{item:C1} $\lim_n \sum_{k \in E}|a_{n,k}|=0$ for all $E \in \mathcal{I}$\textup{;}
\item  \label{item:C2} $\lim_n \sum_k a_{n,k}=1$\textup{;}
\item \label{item:C3} $\lim_n \sum_k |a_{n,k}|=1$\textup{.}
\end{enumerate}
\end{thm}
\begin{proof}
It follows by the classical Silverman--Toeplitz characterization of regular matrices \cite[Chapter 7, Theorem 3]{MR0390692} and \cite[Theorem 22.7]{MR1734462}. (The case where $\mathcal{I}=\mathrm{Fin}$ has been proved by Maddox in \cite[Theorem 1]{MR549529}, 
cf. also \cite[Proposition 4.3]{MR4600193}, 
while the case where $\mathcal{I}$ is the family of asysmptotic density zero sets appears in \cite[Theorem 1]{MR1441452}.)
\end{proof}

Notice that \eqref{eq:inclusioncores} can be rewritten equivalently as $\limsup_n A_n\bm{x} \le \mathcal{I}\text{-}\limsup \bm{x}$ for all $\bm{x} \in \ell_\infty$. As a consequence of Proposition \ref{prop:uniflimsup}, we obtain the \textquotedblleft uniform\textquotedblright\, version of Theorem \ref{thm:connor}. 
\begin{cor}\label{cor:uniforminclusioncores}
Let $\mathcal{I}$ be an ideal on $\omega$ and pick a nonempty family of matrices $\mathscr{A}$ such that $\sup\{\|A\|: A \in \mathscr{A}\}<\infty$. Then 
\begin{equation}\label{eq:newinclusionicores}
\forall \bm{x} \in \ell_\infty, \quad 
\limsup_{n\to \infty}\,  \sup_{A \in \mathscr{A}} A_n\bm{x}
\le \mathcal{I}\text{-}\limsup \bm{x}
\end{equation}
if and only if the following conditions hold\textup{:}
\begin{enumerate}[label={\rm (\textsc{L}\arabic{*})}]
\item \label{item:L1} $\lim_n \sum_{k \in E}|a_{n,k}|=0$ uniformly on $A \in \mathscr{A}$, for all $E \in \mathcal{I}$\textup{;}
\item  \label{item:L2} $\lim_n \sum_k a_{n,k}=1$ uniformly on $A \in \mathscr{A}$\textup{;}
\item \label{item:L3} $\lim_n \sum_k |a_{n,k}|=1$ uniformly on $A \in \mathscr{A}$\textup{.}
\end{enumerate}
\end{cor}
Special cases of Corollary \ref{cor:uniforminclusioncores} can be found in the literature: for instance, if $\mathscr{A}$ is countably infinite and $\mathcal{I}=\mathrm{Fin}$, a direct proof can be found in \cite[Theorem 1]{MR1068009}. 
Choosing $\mathscr{A}=\{F^{\sigma,\nu}A: \nu\in \omega\}$ for a given matrix $A \in (\ell_\infty, \ell_\infty)$, we obtain a more explicit characterization of \cite[Theorem 2]{MR2033032} (for the entries of $F^{\sigma,\nu}A$, see the proof of Theorem \ref{thm:kingmultidimensional}). 

We leave as open question for the interested reader to check whether our main results can be adapted for uniform strong summability, meaning that $\lim_n|A_nx-\eta|=0$ uniformly on $A \in \mathscr{A}$. This would provide a method to obtained streamlined proofs of the results in \cite{Madd1, Madd2}.

\section{Proofs}\label{sec:proofs}
\begin{proof}
[Proof of Theorem \ref{thm:characterizationuniformconvergence}]
As in Section \ref{sec:intro}, we write $\mathscr{A}=\{A^\nu: \nu<\kappa\}$, where each $A^\nu=(A^\nu_{n,k}: n,k \in \omega)$ is a matrix of linear operators in $\mathcal{L}(X,Y)$. Also, ideals are regarded as subset of the Cantor space $\{0,1\}^\omega$, hence we may speak about their topological complexity. 

\medskip

\ref{item:1uniformcondition} $\implies$ \ref{item:2uniformcondition} 
Fix $B \in \mathscr{B}_{\mathscr{A}}$. Then there exists a sequence $(\nu_n: n \in \omega)$ such that $\nu_n<\kappa$ for all $n$, and $B_{n,k}=A_{n,k}^{\nu_n}$ for all $n,k \in \omega$. By hypothesis, for each neighborhood $U$ of $\eta_1$ there exists $S \in \mathcal{I}^\star$ such that $A^\nu_n\bm{x} \in U$ for all $n \in \omega$ and $\nu<\kappa$. In particular, $\sum_kA^{\nu_n}_{n,k}x_k=\sum_k B_{n,k}x_k \in U$ for all $n\in \omega$, hence $\mathcal{I}\text{-}\lim B\bm{x}=\eta_1$. 

\medskip

\ref{item:2uniformcondition} $\implies$ \ref{item:3uniformcondition}  
Let us suppose that there exists a $B,B^\prime \in \mathscr{B}_{\mathscr{A}}$ such that $\eta_B \neq \eta_{B^\prime}$. Since $Y$ is Hausdorff by \cite[Theorem 1.12]{MR1157815}, there exist neighborhoods $U,U^\prime$ of $\eta_B$ and $\eta_{B^\prime}$, respectively, which are disjoint. 
Let $(\nu_n)$ and $(\nu_n^\prime)$ be the sequence of indices associated with $B$ and $B^\prime$, resp.; note that, 
since $\mathcal{I}$ is, in particular, a $F_\sigma$-ideal (hence, not maximal), there exists a partition $\{I,I^\prime\}$ of $\omega$ such that $I,I^\prime \notin \mathcal{I}$. 
Accordingly, let $C \in \mathscr{B}_{\mathscr{A}}$ be the matrix associated with the sequence $(\lambda_n)$ such that $\lambda_n=\nu_n$ if $n\in I$ and $\lambda_n=\nu^\prime_n$ if $n\in I^\prime$. 

If $\eta_C\notin \{\eta_B,\eta_{B^\prime}\}$, there exists a neighborhood $U_{\star}$ of $\eta_C$ such that $U_\star\cap (U\cup U^\prime)=\emptyset$. It follows by construction that 
\begin{displaymath}
\begin{split}
\{n \in \omega: C_n\bm{x} \in U_\star\}\subseteq \{n \in I: B_n\bm{x} \notin U\}\cup \{n \in I^\prime: B^\prime_n\bm{x} \notin U^\prime\} \in \mathcal{I},
\end{split}
\end{displaymath}
which is a contradiction. Otherwise, if $\eta_C=\eta_B$, set $U_\star=U$ so that $\{n \in \omega: C_n\bm{x} \notin U\}\subseteq \{n \in \omega: B_n\bm{x} \notin U\}\cup I^\prime \notin \mathcal{I}\cup \mathcal{I}^\star$, which is again a contradiction. 
The case $\eta_C=\eta_{B^\prime}$ is analogous. 
This proves that $\eta_B=\eta_{B^\prime}$. 

\medskip

\ref{item:3uniformcondition} $\implies$ \ref{item:1uniformcondition} 
Let us suppose that \ref{item:1uniformcondition} fails, so that 
for all $\eta \in Y$ there exists a neighborhood $U$ of $\eta$ such that, for all $S \in \mathcal{I}^\star$, it holds $A_n^\nu \bm{x} \notin U$ for some $\nu<\kappa$ and some $n \in S$. 

Since $\mathcal{I}$ is countably generated by hypothesis, there exists a strictly decreasing sequence of sets $(S_t: t \in \omega)$ such that $S \in \mathcal{I}^\star$ if and only if $S_t\subseteq S$ for some $t \in \omega$. We may also assume without loss of generality that $\min S_t\ge t$ for all $t \in \omega$. 
In addition, since $Y$ is separable we can fix a countable dense subset $\{\eta_m: m \in \omega\}$. It follows that 
\begin{equation}\label{eq:negationitem1}
\forall m \in \omega, \exists U_m, \forall t \in \omega, 
\exists \nu({m,t})<\kappa, \exists n({m,t}) \in S_t, \quad 
 A_{n({m,t})}^{\nu({m,t})}\bm{x} \notin U_m.
\end{equation}
Since $\min S_t\ge t$, it follows that $\lim_t n({m,t})=\infty$ for $m \in \omega$. 
Hence, for each $m$, we can pick a subsequence $(n({m,t_s^{(m)}}): s \in \omega)$ which is strictly increasing. 
Let $f: \omega \to \omega$ be an arbitrary function such that, for each $m \in \omega$, there exist infinitely many $r \in \omega$ for which $f(r)=m$. 
Lastly, choose a function $g: \omega\to \omega$ with the property that $(n(h(r)): r \in \omega)$ is a strictly increasing sequence, where 
$$
\textstyle 
\forall r \in \omega, \quad 
h(r):=\left(f(r), t^{(f(r))}_{g(r)}\right).
$$

Now, suppose for the sake of contradiction that \ref{item:3uniformcondition} holds. 
Let $B$ be the matrix in $\mathscr{B}_{\mathscr{A}}$ associated with the sequence $(\nu_n)$ such that $\nu_n:=\nu(h(r))$ if $n=n(h(r))$ for some $r \in \omega$ and $\nu_n:=0$ otherwise. In particular, 
$$
\forall r\in\omega, \quad 
B_{n(h(r))}=A_{n(h(r))}^{\nu(h(r))}.
$$
%
Since $\{\eta_m: m \in \omega\}$ is dense, there exist $m^\star \in \omega$ and a neighborhood $U$ of $\eta_2$ such that $U\subseteq U_{m^\star}$, where $U_{m^\star}$ is a neighborhood associated to $\eta^\star$ as in \eqref{eq:negationitem1}. It follows 
by $\mathcal{I}\text{-}\lim B\bm{x}=\eta_2$ 
that $\{n \in \omega: B_n\bm{x} \in U\}\in \mathcal{I}^\star$, which implies that 
$$
\exists t_0 \in\omega, \forall n \in S_{t_0}, \quad A^{\nu_n}_n \bm{x} \in U_{m^\star}.
$$
In particular, since $n(m^\star, t) \in S_t$ for all $t \in \omega$ and $(S_t)$ is a decreasing sequence in $\mathcal{I}^\star$, it follows that 
$A^{\nu_{n(m^\star,t)}}_{n(m^\star,t)} \bm{x} \in U_{m^\star}$ for all $t\ge t_0$. 
We obtain by construction that there exists $r^\star \in \omega$ such that $h(r^\star)=(m^\star, t^\star)$, for some $t^\star \ge t_0$. 
We conclude that
$$
A^{\nu_{n(h(r^\star))}}_{n(h(r^\star))}\bm{x} =
A^{\nu(h(r^\star))}_{n(h(r^\star))}\bm{x} =
A^{\nu(m^\star, t^\star)}_{n(m^\star, t^\star)}\bm{x} \in U_{m^\star},
$$
which contradicts \eqref{eq:negationitem1} and concludes the proof. 
\end{proof}


\medskip

\begin{proof}
[Proof of Theorem \ref{thm:uniformtoeplitz}]
\textsc{If part.} Suppose that all items \ref{item:U1}--\ref{item:U5} hold. In particular, items  \ref{item:T1} and \ref{item:T3}--\ref{item:T5} hold for each $A \in \mathscr{A}$. It follows by Theorem \ref{thm:IJREGULARold} that $\{A\}$ is $(\mathcal{I}, \mathcal{J})$-regular with respect to $T$ for each $A \in \mathscr{A}$. In particular, $A \in (c^b(X,\mathcal{I}), c^b(Y, \mathcal{J}))$ for each $A \in \mathscr{A}$. 
The proof that $\mathscr{A}$ satisfies \eqref{eq:uniformAIJregular} for all $\bm{x} \in c^b(X,\mathcal{I})$  along the same lines as the proof of the \textsc{If part} in \cite[Theorem 2.5]{Leo22} with the obvious modifications (we omit further details). 

\medskip

\textsc{Only If part.}  Suppose that the family $\mathscr{A}=\{A^\nu: \nu<\kappa\}$ is $(\mathcal{I}, \mathcal{J})$-regular with respect to $T$. In particular, each $\{A^\nu\}$ is $(\mathcal{I}, \mathcal{J})$-regular with respect to $T$, hence it follows by Theorem \ref{thm:IJREGULARold} that items \ref{item:U1} and \ref{item:U3} hold. 

Since $\mathcal{J}$ is countably generated, there exists a strictly decreasing sequence of sets $(S_t: t \in \omega)$ such that $S \in \mathcal{J}^\star$ if and only if $S_t\subseteq S$ for some $t \in \omega$. In addition, considering that $\mathcal{J}$ is a $F_\sigma$ ideal, there exists a lower semicontinuous submeasure $\varphi: \mathcal{P}(\omega)\to [0,\infty]$ such that 
$$
\mathcal{J}=\{S\subseteq \omega: \varphi(S)<\infty\}, 
$$
see \cite{MR1124539} or \cite[Theorem 1.2.5(a)]{MR1711328}. 
Pick $\bm{x} \in c^b(X,\mathcal{I})$. It follows by 
Theorem \ref{thm:characterizationuniformconvergence} (once we notice that the separability of $Y$ has not been used in the proof of the implications \ref{item:1uniformcondition} $\implies$ \ref{item:2uniformcondition} and \ref{item:2uniformcondition} $\implies$ \ref{item:3uniformcondition}) that 
$$
\forall B \in \mathscr{B}_{\mathscr{A}}, \quad \bm{x}\in \mathrm{dom}(B) \quad \text{ and }\quad 
\mathcal{J}\text{-}\lim B\bm{x}=T(\mathcal{I}\text{-}\lim \bm{x}).
$$

Now, fix $B \in \mathscr{B}_{\mathscr{A}}$. The above observation implies that
$$
B \in (c^b(X,\mathcal{I}), Y^\omega)\quad \text{ and }\quad 
\mathcal{J}\text{-}\lim B\bm{x}=T(\mathcal{I}\text{-}\lim \bm{x})\,\,\text{ for all }\,\,\bm{x} \in c^b(X,\mathcal{I}).
$$
It follows by \cite[Theorem 2.14]{Leo22} (recalling that every countably generated ideal is strongly selective) that $B$ satisfies 
$\|B_{n,\omega}\|<\infty$ for all $n \in \omega$ (which is clear since \ref{item:U1} holds) and, in addition, 
there exists an integer $t(B) \in \omega$ such that 
\begin{equation}\label{eq:existenceTB}
\sup_{n \in S_{t(B)}}\|B_{n,\omega}\|<\infty. 
\end{equation}
Without loss of generality, we can assume that $t(B)$ is the smallest one with this property. 

\begin{claim}\label{claim:tbounded}
The map $t: \mathscr{B}_{\mathscr{A}}\to \omega$ is bounded. 
\end{claim}
\begin{proof}
Suppose for the sake of contradiction that the image $t[\mathscr{B}_{\mathscr{A}}]$ is not bounded, and define 
$$
\forall k \in \omega, \quad 
V_k:=\left\{n \in \omega: k\le \sup_{B\in \mathscr{B}_{\mathscr{A}}} \|B_{n,\omega}\|<k+1 \right\} 
$$
and 
$$
V_\infty:=\left\{n \in \omega: \sup_{B\in \mathscr{B}_{\mathscr{A}}} \|B_{n,\omega}\|=\infty \right\}.  
$$
Note that $\{V_k: k \in \omega\} \cup \{V_\infty\}$ is a partition of $\omega$, and that by construction each 
$$
W_k:=V_\infty \cup \tilde{V}_k, 
\quad \text{ where }\tilde{V}_k:= \bigcup\nolimits_{i\ge k}V_i,
$$
does not belong to $\mathcal{J}$ (hence, $W_k$ is infinite and $\varphi(W_k)=\infty$ for each $k \in \omega$). 

\medskip

\textsc{Case one.} First, suppose that $V_\infty \notin \mathcal{J}$. For each $n \in V_\infty$, there exists an index $\nu_n<\kappa$ such that $\|A^{\nu_n}_{n,\omega}\|\ge n$. Let $B^\star$ be the matrix of linear operators such that the $n$th row of $B^\star$ is the $n$th row of $A^0$ if $n\notin V_\infty$; on the other hand, if $n \in V_\infty$, the $n$th row of $B^\star$ is the $n$th row of $A^{\nu_n}$. It follows by construction that $B^\star \in \mathscr{B}_{\mathscr{A}}$ and, on the other hand, the sequence $(\|B^\star_{n,\omega}\|: n \in \omega)$ is not $\mathcal{J}$-bounded, which contradicts \eqref{eq:existenceTB}. 

\medskip

\textsc{Case two.} Suppose that $V_\infty \in \mathcal{J}$, so that $\tilde{V}_k \notin \mathcal{J}$ for all $k \in \omega$. Define a sequence $(F_k: k \in \omega)$ of finite subsets of $\omega\setminus V_\infty$ as it follows: 
\begin{enumerate}
\item Set $F_0:=\{\min \tilde{V}_0\}$; 
\item Suppose that $F_k$ has been defined for some $k \in \omega$, and let $\iota_k\in\omega$ be the largest integer such that $F_k \cap V_{\iota_k}\neq \emptyset$. 
\item Since $\tilde{V}_{1+\iota_k} \notin \mathcal{J}$, it is possible to fix a finite subset $F_{k+1}\subseteq  \tilde{V}_{1+\iota_k}$ such that $\min F_{k+1}>\max F_k$ and $\varphi(F_{k+1}) \ge k+1$ (this is possible since $\varphi$ is a lower semicontinuous submeasure and $\varphi(\tilde{V}_{1+\iota_k})=\infty$). 
\end{enumerate}
Lastly, define $F:=\bigcup_{k\in \omega}F_k$. Let also $B^\star$ be the matrix of linear of linear operators such that the $n$th row of $B^\star$ is the $n$th row of $A^0$ if $n\notin F$ or $n \in F_0$; on the other hand, if $n \in F_k$ for some nonzero $k \in \omega$, the $n$th row of $B^\star$ is the $n$th row of $A^{\nu_k}$ where $\nu_k<\kappa$ is chosen such that $\|A^{\nu_k}_{n,\omega}\| \ge 1+\iota_{k-1}$. Notice that, since $(\iota_k)$ is strictly increasing, then $\|A^{\nu_k}_{n,\omega}\| \ge k$. 
It follows by construction that $B^\star \in \mathscr{B}_{\mathscr{A}}$ and 
$$
J_k:=\left\{n \in \omega: \|B^\star_{n,\omega}\|\ge k \right\}\supseteq \bigcup\nolimits_{i\ge k}F_i
$$
for all nonzero $k \in \omega$. By the monotonicity of $\varphi$, we get $\varphi(J_k)=\infty$ for each $k>0$, i.e., $J_k \notin \mathcal{J}$. This contradicts \eqref{eq:existenceTB} which states that the sequence $(\|B_{n,\omega}^\star\|: n \in \omega)$ is $\mathcal{J}$-bounded. 
\end{proof}

\medskip

Since $t$ is bounded by Claim \ref{claim:tbounded}, it follows that there exists $m \in \omega$ such that $\sup_{n \in S_m}\|B_{n,\omega}\|<\infty$ for all $B \in \mathscr{B}_{\mathscr{A}}$. In particular, since $S_m \in \mathcal{J}^\star$, item \ref{item:U2} holds. 

Lastly, choosing the constant sequences $(x,x,\ldots) \in c^b(X,\mathcal{I})$ and each $\bm{x} \in c_{00}^b(X,\mathcal{I})$, we obtain by the definition of $(\mathcal{I}, \mathcal{J})$-regularity of $\mathscr{A}$ with respect to $T$ that items \ref{item:U4} and \ref{item:U5} hold, respectively. This finishes the proof. 
\end{proof}

\medskip

\begin{proof}
[Proof of Theorem \ref{cor:maincharacterizationIJregular}]
Endow both $X=\mathbf{R}^d$ and $Y=\mathbf{R}^m$ with the corresponding $1$-norms (so that $\|x\|=\sum_i |x_i|$ for all $x \in X$). 
Then it is not difficult to check that $\|A_{n,k}\|=\max_j \sum_i|a_{n,k}(i,j)|$ for all $n,k \in \omega$. Since $X$ is an abstract $L$-space (i.e., $X$ is Banach lattice with $\|x+y\|=\|x\|+\|y\|$ for all $x,y \in X$ with $x,y \ge 0$), it follows, as in the proof of \cite[Corollary 2.11]{Leo22}, that 
$$
\frac{1}{d}\sum\nolimits_{k \in E}\sum\nolimits_{i,j}|a_{n,k}(i,j)| \le \|A_{n,E}\|
=\sum\nolimits_{k \in E}\|A_{n,k}\|
\le \sum\nolimits_{k \in E}\sum\nolimits_{i,j}|a_{n,k}(i,j)|
$$
for all $n \in \omega$ and $E\subseteq \omega$.
We obtain that items \ref{item:D1} and \ref{item:D2} are equivalent to items \ref{item:U1} and \ref{item:U2}, respectively, cf. the proof of \cite[Theorem 2.8]{Leo25}. 
Thanks to \cite[Theorem 3.8]{Leo22}, we obtain that item \ref{item:U1} is equivalent to $A^\nu \in (\ell_\infty(X), \ell_\infty(Y))$ for each $\nu<\kappa$ in finite dimension, hence item \ref{item:D1} implies item \ref{item:U3}. 
In addition, it is easy to see that item \ref{item:D3} is equivalent to item \ref{item:U4}. 

At this point, we recall that the unit ball of $\ell_\infty(\mathbf{R}^d)$ coincides with the norm closure of the convex hull of its extreme points, thanks to \cite[Proposition 3.1]{Leo25} (the one-dimensional case can be found in \cite{Goodner64}). Taking into account Theorem \ref{thm:characterizationuniformconvergence} above, we obtain by \cite[Corollary 2.6]{Leo25} that item \ref{item:U5} is equivalent to: 
\begin{enumerate}[label={\rm (\textsc{U}\arabic{*}$^\flat$)}] 
\setcounter{enumi}{4}
\item \label{item:U5flat} $\mathcal{J}\text{-}\lim_n \sum_{k \in E}B_{n,k}x_k=0$ for all $B \in \mathscr{B}_{\mathscr{A}}$, all $E \in \mathcal{I}$, and all $\bm{x} \in \mathrm{Ext}(B_X)^\omega$\textup{.} 
\end{enumerate}
(Here, as usual, $\mathrm{Ext}(S)$ stands for the set of extreme points of $S\subseteq X$.) Let $\{B^\mu: \mu<\tau\}$ be an enumeration of $\mathscr{B}_{\mathscr{A}}$, and write $B^\mu=[\,b^\mu_{n,k}(i,j): 1\le i\le m, 1\le j\le d\,]$ for each $\mu<\tau$. 
Now, pick $E \in \mathcal{I}$, an index $\mu<\tau$, let $e_t:=(0,\ldots,0,1,0,\ldots,0) \in \mathbf{R}^d$ be the vector which is $1$ at the $t$th component and $0$ otherwise for each $t \in \{1,\ldots,d\}$, and suppose that $\mathcal{J}\text{-}\lim_n \sum_{k \in E}B_{n,k}^\mu x_k=0$ for all sequences $(x_k: k \in \omega)$ taking values in $\mathrm{Ext}(B_X)=\{\pm e_t: t \in \{1,\ldots,d\}\}$. This is equivalent to 
$
\mathcal{J}\text{-}\lim_{n} \sum\nolimits_{k \in E}\sum\nolimits_j b_{n,k}(i,j)^\mu x_k(j)=0
$ 
for all $i\in \{1,\ldots, m\}$ and all sequences $(x_k: k \in\omega)$ with values in $\{\pm e_t: t \in \{1,\ldots,d\}\}$. 
Fix also indices $i_0 \in \{1,\ldots,m\}$ and $j_0 \in \{1,\ldots,d\}$ and choose the constant sequence $(x_k: k \in \omega)$ with value  $e_{j_0}$. It follows that $\mathcal{J}\text{-}\lim_{n} \sum\nolimits_{k \in E} b_{n,k}(i_0,j_0)^\mu=0$.
Therefore item \ref{item:U5flat} turns out to be equivalent to: 
\begin{enumerate}[label={\rm (\textsc{D}\arabic{*}$^\natural$)}] 
\setcounter{enumi}{3}
\item \label{item:D4flat} $\mathcal{J}\text{-}\lim_n \sum_{k \in E}b^\mu_{n,k}(i,j)=0$ for all integers $1\le i\le m$, $1\le j\le d$, all sets $E \in \mathcal{I}$, and all indices $\mu<\tau$\textup{.} 
\end{enumerate}
Applying Theorem \ref{thm:characterizationuniformconvergence} once we fix a set $E$ and integers $i,j$, we obtain that item \ref{item:D4flat} is equivalent to item \ref{item:D4}. To sum up, \ref{item:U5} is equivalent to \ref{item:D4}. 
The conclusion follows by the characterization given in Theorem \ref{thm:uniformtoeplitz}.
\end{proof}

\medskip

\begin{proof}
[Proof of Corollary \ref{cor:maincorollaryuniform}] 
This follows by Theorem \ref{cor:maincharacterizationIJregular} by setting $d=m=1$.
\end{proof}

\medskip

\begin{proof}
[Proof of Corollary \ref{cor:maptozero}]
Let $\mathcal{I}$ be a maximal ideal on $\omega$, so that $c^b(X,\mathcal{I})=\ell_\infty(X)$. Let also $T$ be the zero operator. As in the other proof, denote $\mathscr{B}_{\mathscr{A}}=\{B^\mu: \mu<\tau\}$. 

\medskip

\textsc{If part.} 
This is obvious since \ref{item:D3sharp} 
implies \ref{item:D3} with $t(i,j)=0$ for all $1\le i\le m$ and $1\le j\le d$. 

\medskip

\textsc{Only If part.} 
Observe that $\mathscr{A}$ (hence also $\{A^\nu\}$ for every $\nu<\kappa$) is $(\mathcal{I}, \mathcal{J})$-regular with respect to $T$. It follows by Theorem \ref{cor:maincharacterizationIJregular} that items \ref{item:D1} and \ref{item:D2} hold. 
In addition, by the uniform condition in \eqref{eq:uniformconditionwrgwg}, $\{B^\mu\}$ is $(\mathcal{I}, \mathcal{J})$-regular with respect to $T$ for every $\mu<\tau$. Thanks to \cite[Corollary 2.12]{Leo22}, we obtain that 
$$
\mathcal{J}\text{-}\lim_{n\to \infty} \sum\nolimits_k \sum\nolimits_{i,j} |b^\mu_{n,k}(i,j)|=0 \quad \text{ uniformly on }\mu.
$$
Fix integers $1\le i\le m$ and $1\le j\le d$ and consider the positive and negative part of each $b^\mu_{n,k}(i,j)$, so that $\mathcal{J}\text{-}\lim_{n} \sum\nolimits_k (b^\mu_{n,k}(i,j))^+=0$ uniformly on $\mu$ (and similarly for the negative part). It follows by Theorem \ref{thm:characterizationuniformconvergence} that $\mathcal{J}\text{-}\lim_{n} \sum\nolimits_k (a^\nu_{n,k}(i,j))^+=0$ uniformly on $\nu$ (and similarly for the negative part). Therefore also item \ref{item:D3sharp} 
holds. 
\end{proof}

\medskip

\begin{proof}
[Proof of Theorem \ref{thm:kingmultidimensional}]
First of all, note that if $A^\nu:=F^{\sigma,\nu}A=[a^\nu_{n,k}(i,j): 1\le i\le m, 1\le j\le d, \text{ and }n,k \in \omega]$ then $a^\nu_{n,k}(i,j):=\frac{1}{n+1}\sum_{h=0}^n a_{\sigma(\nu+h),k}(i,j)$. This follows by the identity
\begin{displaymath}
\begin{split}
A^\nu_n \bm{x}&=F^{\sigma,\nu} (A\bm{x})=\frac{1}{n+1}\sum_{h=0}^n A_{\sigma(\nu+h)}\bm{x}\\
&=\frac{1}{n+1}\sum_{h=0}^n \sum_{k\in \omega} A_{\sigma(\nu+h),k}x_k
=\frac{1}{n+1}\sum_{k \in \omega} \sum_{h=0}^n A_{\sigma(\nu+h),k}x_k.
\end{split}
\end{displaymath}
Taking also into account the observation preceding the statement of 
Theorem \ref{thm:kingmultidimensional}, 
we obtain that items \ref{item:K2} and \ref{item:K3} are a rewriting of items \ref{item:D3} and \ref{item:D4}, respectively. 

\medskip

\textsc{If part.} Suppose that item \ref{item:K1} holds and observe that, for all $\nu \in \omega$ and all integers $1\le i\le m$ and $1\le j\le d$, we have 
\begin{displaymath}
\begin{split}
\sum\nolimits_{k \in \omega} |a^\nu_{n,k}(i,j)| 
&\le \frac{1}{n+1}\sum\nolimits_{k\in \omega}\sum\nolimits_{h\in n+1} |a_{\sigma(\nu+h),k}(i,j)|\\
&\le \max_{h \in n+1}\,\sum\nolimits_{k\in \omega}|a_{\sigma(\nu+h),k}(i,j)| 
\le \sup\nolimits_n\sum\nolimits_{k\in \omega} |a_{n,k}(i,j)|.
\end{split}
\end{displaymath}
We obtain 
$$
\sup_{\nu \in \omega} \, \sup_{n \in \omega} 
\, \sum\nolimits_{k\in \omega} \sum\nolimits_{i,j} |a^\nu_{n,k} (i,j)|<\infty,
$$
which implies both items \ref{item:D1} and \ref{item:D2}. The conclusion follows by Theorem \ref{cor:maincharacterizationIJregular}.  

\medskip

\textsc{Only If part.} Suppose that $A$ is $(\mathcal{I}, \mathcal{J}, \sigma)$-almost regular with respect to $T$. Thanks to Theorem \ref{cor:maincharacterizationIJregular} and the observations above, both \ref{item:K2} and \ref{item:K3} hold. Lastly, observe that $A$ belongs, in particular, to the matrix class $(c(X), \ell_\infty(Y))$. It follows by \cite[Corollary 3.11]{Leo22} that the latter is equivalent to item \ref{item:K1}, cf. also \cite[Equation (5.3)]{Leo22}.  
\end{proof}

\medskip

\begin{proof}
[Proof of Proposition \ref{prop:uniflimsup}]
Fix $\bm{x} \in \ell_\infty$ and let $\eta_L$ and $\eta_R$ be the values on the left hand side and right hand side of \eqref{eq:claimuniflimsup}. First of all, for each $B \in \mathscr{B}_{\mathscr{A}}$ it is clear that $B_n\bm{x}\le \sup \{A_n\bm{x}: A \in \mathscr{A}\}$. It follows that $\mathcal{I}\text{-}\limsup B\bm{x}\le \eta_L$. By the arbitrariness of $B$, we obtain $\eta_R\le \eta_L$. 

Conversely, observe that $\eta_L$ is the maximal $\mathcal{I}$-cluster point of the sequence $(\sup \{A_n\bm{x}: A \in \mathscr{A}\}: n \in \omega)$. Hence there exists a function $\alpha: \omega\to \omega$ such that 
$$
\forall \varepsilon>0, \quad \{n \in \omega: A^{\alpha(n)}_n\bm{x}\ge \eta_L-\varepsilon\} \notin \mathcal{I}. 
$$
Pick a matrix $B^\star=(b^\star_{n,k}: n,k \in \omega) \in \mathscr{B}_{\mathscr{A}}$ so that $b^\star_{n,k}=a_{n,k}^{\alpha(n)}$ for each $n \in \omega$. It follows by construction that $\mathcal{I}\text{-}\limsup_n B^\star_n\bm{x} \ge \eta_L$. Therefore $\eta_R \ge \eta_L$. 
\end{proof}

\medskip

\begin{proof}
[Proof of Corollary \ref{cor:uniforminclusioncores}]
Thanks to Proposition \ref{prop:uniflimsup}, condition \eqref{eq:newinclusionicores} can be rewritten equivalently as 
$$
\forall B \in \mathscr{B}_{\mathscr{A}}, \forall \bm{x} \in \ell_\infty, \quad 
\limsup B\bm{x} \le \mathcal{I}\text{-}\limsup \bm{x}.
$$
Applying Theorem \ref{thm:connor} at each $B \in  \mathscr{B}_{\mathscr{A}}$, it follows that \eqref{eq:newinclusionicores} holds if and only if each matrix $B$ satisfies items \ref{item:C1}--\ref{item:C3}. The latter ones are equivalent to items \ref{item:L1}--\ref{item:L3}, thanks to Theorem \ref{thm:characterizationuniformconvergence}. 
\end{proof}

\subsection*{Acknowledgments} 
The author is grateful to an anonymous referee for several helpful comments and a careful reading of the manuscript.


\bibliographystyle{amsplain}

\begin{thebibliography}{99}

\bibitem{MR4600193}
A.~Aveni and P.~Leonetti, \emph{Most numbers are not normal}, Math. Proc.
  Cambridge Philos. Soc. \textbf{175} (2023), no.~1, 1--11. 

\bibitem{MR3543775}
M.~Balcerzak, S.~G\l{a}b, and J.~Swaczyna, \emph{Ideal invariant injections},
  J. Math. Anal. Appl. \textbf{445} (2017), no.~1, 423--442. 

\bibitem{MR310489}
H.~T. Bell, \emph{Order summability and almost convergence}, Proc. Amer. Math.
  Soc. \textbf{38} (1973), 548--552. 

\bibitem{MR427890}
H.~T. Bell, \emph{{$S$}-limits and {$A$}-summability}, Proc. Amer. Math. Soc.
  \textbf{61} (1976), no.~1, 49--53 (1977). 

\bibitem{MR1817226}
J.~Boos, \emph{Classical and modern methods in summability}, Oxford
  Mathematical Monographs, Oxford University Press, Oxford, 2000, Assisted by
  Peter Cass, Oxford Science Publications. 

\bibitem{MR2209588}
C.~{\c C}akan and B.~Altay, \emph{Statistically boundedness and statistical
  core of double sequences}, J. Math. Anal. Appl. \textbf{317} (2006), no.~2,
  690--697. 

\bibitem{MR1734462}
J.~Connor, \emph{A topological and functional analytic approach to statistical
  convergence}, Analysis of divergence ({O}rono, {ME}, 1997), Appl. Numer.
  Harmon. Anal., Birkh\"{a}user Boston, Boston, MA, 1999, pp.~403--413.
  

\bibitem{MR2241135}
J.~Connor, J.~A. Fridy, and C.~Orhan, \emph{Core equality results for
  sequences}, J. Math. Anal. Appl. \textbf{321} (2006), no.~2, 515--523.
 

\bibitem{MR2033032}
K.~Demirci and S.~Yardimci, \emph{{$\sigma$}-core and {$\mathscr{I}$}-core of
  bounded sequences}, J. Math. Anal. Appl. \textbf{290} (2004), no.~2,
  414--422. 

\bibitem{MR1711328}
I.~Farah, \emph{Analytic quotients: theory of liftings for quotients over
  analytic ideals on the integers}, Mem. Amer. Math. Soc. \textbf{148} (2000),
  no.~702, xvi+177. 

\bibitem{MR1441452}
J.~A. Fridy and C.~Orhan, \emph{Statistical core theorems}, J. Math. Anal.
  Appl. \textbf{208} (1997), no.~2, 520--527. 

\bibitem{MR1416085}
J.~A. Fridy and C.~Orhan, \emph{Statistical limit superior and limit inferior}, Proc. Amer.
  Math. Soc. \textbf{125} (1997), no.~12, 3625--3631. 

\bibitem{Goodner64}
D.~B. Goodner, \emph{The closed convex hull of certain extreme points}, Proc.
  Amer. Math. Soc. \textbf{15} (1964), 256--258. 

\bibitem{MR282142}
W.~B. Jurkat and A.~Peyerimhoff, \emph{Fourier effectiveness and order
  summability}, J. Approximation Theory \textbf{4} (1971), 231--244.
 

\bibitem{MR4126774}
V.~Kadets and D.~Seliutin, \emph{On relation between the ideal core and ideal
  cluster points}, J. Math. Anal. Appl. \textbf{492} (2020), no.~1, 124430, 7.
 

\bibitem{MR4042592}
M.~Karaku{\c{s}} and F.~Ba{\c{s}}ar, \emph{Operator valued series, almost
  summability of vector valued multipliers and (weak) compactness of summing
  operator}, J. Math. Anal. Appl. \textbf{484} (2020), no.~1, 123651, 16.


\bibitem{MR201872}
J.~P. King, \emph{Almost summable sequences}, Proc. Amer. Math. Soc.
  \textbf{17} (1966), 1219--1225.

\bibitem{Leo25}
P.~Leonetti, \emph{Regular matrices of unbounded linear operators--{II}}, manuscript.

\bibitem{Leo22}
P.~Leonetti,  \emph{Regular matrices of unbounded linear operators}, Proc. Roy. Soc.
  Edinburgh Sect. A, to appear
  (\href{https://doi:10.1017/prm.2024.1}{doi:10.1017/prm.2024.1}).

\bibitem{MR3955010}
P.~Leonetti,  \emph{Characterizations of the ideal core}, J. Math. Anal. Appl.
  \textbf{477} (2019), no.~2, 1063--1071.

\bibitem{MR3920799}
P.~Leonetti and F.~Maccheroni, \emph{Characterizations of ideal cluster
  points}, Analysis (Berlin) \textbf{39} (2019), no.~1, 19--26. 

\bibitem{MR27868}
G.~G. Lorentz, \emph{A contribution to the theory of divergent sequences}, Acta
  Math. \textbf{80} (1948), 167--190. 

\bibitem{MR0390692}
I.~J. Maddox, \emph{Elements of functional analysis}, Cambridge University
  Press, London-New York, 1970.

\bibitem{MR0447877}
I.~J. Maddox, \emph{Matrix maps of bounded sequences in a {B}anach space}, Proc.
  Amer. Math. Soc. \textbf{63} (1977), no.~1, 82--86. 

\bibitem{Madd1}
I.~J. Maddox, \emph{A new type of convergence}, Math. Proc. Cambridge Philos. Soc.
  \textbf{83} (1978), 61--64.

\bibitem{Madd2}
I.~J. Maddox, \emph{On strong almost convergence}, Math. Proc. Cambridge Philos.
  Soc. \textbf{104} (1979), no.~2, 345--350.

\bibitem{MR549529}
I.~J. Maddox, \emph{Some analogues of {K}nopp's core theorem}, Internat. J. Math.
  Math. Sci. \textbf{2} (1979), no.~4, 605--614.

\bibitem{MR568707}
I.~J. Maddox, \emph{Infinite matrices of operators}, Lecture Notes in Mathematics,
  vol. 786, Springer, Berlin, 1980. 

\bibitem{MR259482}
S.~M. Mazhar and A.~H. Siddiqi, \emph{On the {$F\sb{A}$} and {$A\sb{B}$}
  summability of a trigonometric sequence}, Indian J. Math. \textbf{9} (1967),
  461--466. 

\bibitem{MR1124539}
K.~Mazur, \emph{{$F_\sigma$}-ideals and {$\omega_1\omega_1^*$}-gaps in the
  {B}oolean algebras {$P(\omega)/I$}}, Fund. Math. \textbf{138} (1991), no.~2,
  103--111. 

\bibitem{MR948914}
F.~M\'{o}ricz and B.~E. Rhoades, \emph{Almost convergence of double sequences
  and strong regularity of summability matrices}, Math. Proc. Cambridge Philos.
  Soc. \textbf{104} (1988), no.~2, 283--294.

\bibitem{MR1068009}
C.~Orhan, \emph{Sublinear functionals and {K}nopp's core theorem}, Internat. J.
  Math. Math. Sci. \textbf{13} (1990), no.~3, 461--468. 

\bibitem{MR154005}
R.~A. Raimi, \emph{Invariant means and invariant matrix methods of
  summability}, Duke Math. J. \textbf{30} (1963), 81--94. 

\bibitem{MR1157815}
W.~Rudin, \emph{Functional analysis}, second ed., International Series in Pure
  and Applied Mathematics, McGraw-Hill, Inc., New York, 1991. 

\bibitem{MR235340}
P.~Schaefer, \emph{Almost convergent and almost summable sequences}, Proc.
  Amer. Math. Soc. \textbf{20} (1969), 51--54.

\bibitem{MR306763}
P.~Schaefer, \emph{Infinite matrices and invariant means}, Proc. Amer. Math. Soc.
  \textbf{36} (1972), 104--110. 

\bibitem{MR358128}
M.~Stieglitz, \emph{Eine {V}erallgemeinerung des {B}egriffs der
  {F}astkonvergenz}, Math. Japon. \textbf{18} (1973), 53--70. 

\end{thebibliography}

\providecommand{\href}[2]{#2}

\end{document}